\documentclass[a4paper,12pt]{amsart}

\usepackage[english]{babel}
\usepackage{url}
\usepackage{xcolor}
\usepackage[all,cmtip]{xy}
\usepackage{graphicx}

\usepackage{amsmath,amssymb,enumerate,mathrsfs}

\newtheorem{thm}{Theorem}[section]
\newtheorem{corollary}[thm]{Corollary}
\newtheorem{proposition}[thm]{Proposition}
\newtheorem{lemma}[thm]{Lemma}
\newtheorem{step}[thm]{Step}

\theoremstyle{definition}

\newtheorem{remark}[thm]{Remark}

\newtheorem{example}[thm]{Example}

\numberwithin{equation}{section}

\newcommand{\NE}{\operatorname{NE}}
\newcommand{\Pic}{\operatorname{Pic}}

\newcommand{\Exc}{\operatorname{Exc}}
\newcommand{\Hom}{\operatorname{Hom}}

\newcommand{\ph}{\varphi}
\newcommand{\w}{\widetilde}
\newcommand{\ma}{\mathcal}
\newcommand{\la}{\longrightarrow}
\newcommand{\ol}{\mathcal{O}}

\newcommand{\pr}{\mathbb{P}}

\newcommand{\R}{\mathbb{R}}
\newcommand{\Z}{\mathbb{Z}}
\newcommand{\N}{\mathcal{N}_1}

\newcommand{\im}{\operatorname{Im}}

\title[Classification of Fano 4-folds with $\delta=3$ and $\rho=5$]{Classification of Fano 4-folds with Lefschetz defect 3 and Picard number 5}
\author{Cinzia Casagrande}
\address{Cinzia Casagrande:
 Universit\`a di Torino,
 Dipartimento di Matematica,
via Carlo Alberto 10,
 10123 Torino - Italy}
\email{cinzia.casagrande@unito.it}
\author{Eleonora A. Romano}
\address{Eleonora A. Romano: Universytet Warszawski, Instytut Matematyki, Banacha 2, 02-097 Warszawa - Poland}
\email{elrom@mimuw.edu.pl}

\calclayout

\begin{document}

\maketitle

\section{Introduction}
\noindent 
 The classification of (smooth, complex) Fano manifolds has been achieved up to dimension $3$ and attracts a lot of attention also in higher dimensions, especially due to the Minimal Model Program. Indeed we recall that Fano manifolds appear in the birational classification of varieties of negative Kodaira dimension: in this case the MMP is expected to end up with a fiber type morphism whose fibers are (mildly singular) Fano varieties. 

In the early 80's the classification of Fano 3-folds in \cite{morimukai} due to Mori and Mukai was the starting point to study Fano manifolds via their contractions. In fact, the Fano condition makes the situation special, because the Cone and the Contraction Theorems hold for the whole cone of effective curves.
Nevertheless, there is still no complete classification of Fano varieties in dimension $4$ and higher.

In this paper we focus on classification results of some Fano 4-folds.
Let us fix some notation. Given a smooth complex projective variety $X$, we denote by $\mathcal{N}_{1}(X)$ the $\mathbb{R}$-vector space of one-cycles with real coefficients, modulo numerical equivalence, whose dimension is the \textit{Picard number} $\rho_{X}$.

Let $D$ be a prime divisor in $X$. The inclusion $i\colon D\hookrightarrow X$ induces a pushforward of one-cycles $i_{*}\colon \mathcal{N}_{1}(D)\to \mathcal{N}_{1}(X)$. We set $\mathcal{N}_{1}(D,X):=i_{*}(\mathcal{N}_{1}(D))\subseteq \mathcal{N}_{1}(X)$, which is the linear subspace of $\mathcal{N}_{1}(X)$ spanned by numerical classes of curves contained in $D$. 
In \cite{codim} the following invariant, called \textit{Lefschetz defect}, was introduced: 
\begin{center}
$\delta_{X}:=\max\{\operatorname{codim} \mathcal{N}_{1}(D,X)| D\subset X \text{ prime divisor} \}$. 
\end{center}
By [loc.\ cit., Th.~1.1], if $X$ is a 
Fano manifold of arbitrary dimension with $\delta_X\geq 4$, then $X\cong S\times T$, with $S$ a del Pezzo surface. As a consequence, all Fano $4$-folds with $\delta_X\geq 4$ are well known, being product of two del Pezzo surfaces. 

In this paper we deal with the case in which $X$ is a Fano $4$-fold with $\delta_{X}=3$. 
Under this assumption, by \cite[Th.~1.1]{cdue} we know that if $X$ is not a product of two del Pezzo surfaces, then $\rho_{X}\in \{5,6\}$. Therefore in order to complete the classification of Fano $4$-folds with $\delta_X=3$ we are left to study the cases in which $\rho_{X}=5$ and $\rho_{X}=6$. This setting has already been studied in \cite{M_R}, where several properties of these $4$-folds are shown (see Th.~\ref{factorization}).

Our main result is the complete classification of Fano $4$-folds with $\rho_{X}=5$ and $\delta_X=3$. To give the statement, we first need to introduce two examples; the former is due to Tsukioka, while the latter is new.
\begin{example}[\cite{toru5}]\label{e1}
Let $p,q\in\pr^4$ be distinct points, and $Q\subset \pr^4$ a smooth quadric surface disjoint from the line $\overline{pq}$. 
Let $Z$ be the blow-up of $\pr^4$ along $\overline{pq}$,  $F_p,F_q\subset Z$  the fibers over $p$ and $q$ respectively, and $S\subset Z$ the transform of $Q$. 

The surfaces $S,F_p,F_q$ are pairwise disjoint in $Z$; let
 $X$ be the blow-up of $Z$ along $S$, $F_p$, and $F_q$. Then $X$ is a Fano $4$-fold with $\rho_X=5$ and $\delta_X=3$.

We can also describe $X$ as follows:
let $X'$ be the blow-up of $\pr^4$ along $p$, $q$, and $Q$. Then $X$ is the blow-up of $X'$ along the transform of the line $\overline{pq}$. 
\end{example}
\begin{example}\label{e2}
 The $4$-fold $Z:=\pr_{\pr^2}(\ol^{\oplus 2}\oplus \ol(2))$ has 
a divisorial contraction $Z\to W:=\pr(1,1,1,2,2)$ which sends 
 the exceptional divisor $D$ to a curve (the singular locus of $W$); let  $F_1,F_2\subset D$ be two distinct fibers.

Let moreover $\ol_W(1)$ be the ample generator of $\Pic(W)$, and
  ${H}\in\Pic(Z)$ the pullback of $\ol_W(1)$. Consider $S\subset Z$ a general complete intersection of elements in the linear systems $|{H}|$ and $|2{H}|$;  $S$ is a del Pezzo surface of degree $2$ (see \S\ref{details}).

The surfaces $S,F_1,F_2$ are pairwise disjoint in $Z$; let
 $X$ be the blow-up of $Z$ along $S$, $F_1$, and $F_2$.
Then $X$ is a  Fano $4$-fold with $\rho_X=5$ and $\delta_X=3$ (see Lemma \ref{e2fano}).
\end{example}
We recall that toric Fano $4$-folds have been classified by Batyrev \cite{bat2}
and Sato \cite{sato}; we refer to \cite{bat2} for the 
 terminology concerning toric varieties and their combinatorial type.
\begin{thm} \label{main} Let $X$ be a Fano 4-fold with $\rho_X=5$ and $\delta_X=3$. Then $X$ is one of the following varieties: the toric Fano $4$-folds $K_1$, $K_2$, $K_3$, $K_4\cong\pr^2\times S_4$ ($S_4$  the blow-up of $\pr^2$ in three non collinear points), or one of the $4$-folds of Examples \ref{e1} and \ref{e2}.
\end{thm}
We describe these varieties and their invariants in \S \ref{explicit}, see Table \ref{table}. They are all rational, as already shown in  \cite[Cor.~1.3]{M_R}. The first five are rigid, while Example \ref{e2} yields a positive dimensional family (see Lemma \ref{e2fano}).

\medskip

We note that the assumptions $\rho_X=5$ and $\delta_X=3$ imply that $X$ contains a prime divisor $D$ with $\dim\N(D,X)=2$; in fact it is easy to see that all the varieties listed in Theorem \ref{main} also contain a prime divisor $D'$ with $\rho_{D'}=2$. We obtain the following application to Fano $4$-folds containing a prime divisor with $\rho=2$; an analogous result for the case of a prime divisor with $\rho=1$ is given in \cite[Th.~3.8]{minimal}.
\begin{corollary}\label{rho2}
Let $X$ be a Fano $4$-fold containing a prime divisor $D$ with $\rho_D=2$, or more generally with $\dim\N(D,X)=2$. Then either $X\cong\pr^2\times S$, or $\rho_X\leq 5$. Moreover, $\rho_X=5$ if and only if $X$ is one of the $4$-folds listed in Theorem \ref{main}. 
\end{corollary}

\medskip

Corollary \ref{rho2} is also related to the study of Fano $4$-folds having an elementary divisorial contraction 
$\sigma\colon X\to X'$ such that $\sigma(\Exc (\sigma))$ is a curve, because then
 automatically  $\dim\N(\Exc(\sigma),X)=2$. It follows from \cite{codim} that $\rho_X\leq 5$, and Tsukioka has classified the case $\rho_X=5$ when $\sigma$ is a smooth blow-up, as follows.
\begin{thm}[\cite{toru5}]\label{toru}
Let $X$ be a Fano $4$-fold obtained as a blow-up $\sigma\colon X\to X'$ of a smooth, irreducible curve in a smooth $4$-fold $X'$, and assume that $\rho_X=5$. Then $X$ is either the toric $4$-fold $K_3$, or Example \ref{e1}. In both cases $\Exc(\sigma)\cong\pr^1\times\pr^2$, $\ma{N}_{\Exc(\sigma)/X}\cong\ol_{\pr^1\times\pr^2}(-1,-1)$, and $X'$ is not Fano.
\end{thm}
In fact this result is shown in \cite{toru5} in arbitrary dimension $\geq 4$.
Here we extend the classification as follows.
\begin{corollary}\label{31}
Let $X$ be a Fano $4$-fold having an elementary divisorial contraction 
$\sigma\colon X\to X'$ such that $\sigma(\Exc(\sigma))$
is a curve, and assume that $\rho_X=5$. Then $X$ is one of the toric Fano $4$-folds $K_1$, $K_3$, or one of the $4$-folds of Examples \ref{e1} and \ref{e2}. In all cases, $\Exc(\sigma)\cong\pr^1\times\pr^2$, and $X'$ is not Fano.

In cases $K_1$ and  Example \ref{e2}, we have $\ma{N}_{\Exc(\sigma)/X}\cong\ol_{\pr^1\times\pr^2}(-1,-2)$, and $\sigma(\Exc(\sigma))$ is a curve of singular points.
\end{corollary}

\medskip

 Let us briefly discuss the strategy used to show Theorem \ref{main}.
We build on results from \cite{codim,eleonora,M_R}, which give a structure theorem for Fano $4$-folds $X$ with $\rho_X=5$ and $\delta_X=3$. More precisely, $X$ has always a flat fibration $X\to\pr^2$, that factors as $X\to Y\to\pr^2$, where the first map is a conic bundle, and the second one a $\pr^1$-bundle. We collect these results in Theorem \ref{factorization}.

We show that the fibration $X\to\pr^2$ has also a different factorization as $X\to Z\stackrel{\ph}{\to}\pr^2$, where $\ph$ is a $\pr^2$-bundle, and $X\to Z$ is the blow-up of three pairwise disjoint smooth surfaces $S_i\subset Z$, horizontal for $\ph$. When all the surfaces $S_i$ are sections of $\ph$, it turns out that $X$ is toric. Otherwise, we prove that two surfaces $S_i$ are always sections, and the third one has degree $2$ over $\pr^2$. In this case, we show that $X$ is one of Examples \ref{e1} or \ref{e2}.
\subsection*{Acknowledgments} We would like to thank Gianluca Occhetta and Luis E. Sol{\'a} Conde for many valuable discussions. The second author has been supported by the Polish National Science Center grant 2016/23/G/ST1/04282, and is grateful to the University of Torino for the kind hospitality and support provided during part of the preparation of this work.
\subsection*{Notations}
\noindent We work over the field of complex numbers. Let $X$ be a smooth projective variety. 

$\sim$ denotes linear equivalence for divisors.

$\mathcal{N}_{1}(X)$ is the $\mathbb{R}$-vector space of one-cycles with real coefficients, modulo numerical equivalence, and
 $\dim \mathcal{N}_{1}(X)=\rho_{X}$ is the Picard number of $X$.

We denote by $[C]$ the numerical equivalence class in $\mathcal{N}_{1}(X)$
of a one-cycle $C$ of $X$. 

 $\operatorname{NE}(X)\subset \mathcal{N}_{1}(X)$ is the convex cone generated by classes of effective curves. 

A \textit{contraction} of $X$ is a surjective morphism $\varphi\colon X\to Y$ with connected fibers, where $Y$ is normal and projective.
 
The \textit{relative cone} $\text{NE}(\varphi)$ of $\varphi$ is the convex subcone of $\text{NE}(X)$ generated by classes of curves contracted by $\varphi$.

A \textit{conic bundle} $h\colon X\to Y$ is a  fiber type contraction such that every fiber is one-dimensional and $-K_X$ is $h$-ample; then every fiber is isomorphic to a plane conic. 
\section{Proof of the main result}
\subsection{Preliminaries}
\noindent We begin by collecting in a unique statement the known results on the structure of Fano $4$-folds $X$ with $\rho_X=5$ and $\delta_X=3$. This is our starting point to prove Theorem \ref{main}.
\begin{thm}[\cite{codim,eleonora,M_R}] \label{factorization} 
Let $X$ be a Fano $4$-fold with $\rho_X=5$ and $\delta_X=3$.
Then there exists a diagram:
$$X\stackrel{f}{\la} X_{2}\stackrel{\psi}{\la} Y\stackrel{\xi}{\la}\pr^2$$
with the following properties:
\begin{enumerate}[$(a)$]
\item \label{Y} $Y\cong\pr_{\pr^2}(\ol\oplus\ol(a))$ with $a\in\{0,1,2\}$, and $\xi$ is the natural $\pr^1$-bundle;
\item \label{g} $\psi$ is a $\pr^1$-bundle;
\item \label{blowup} $f$ is the blow-up of two disjoint smooth, irreducible surfaces $B_1,B_2\subset X_2$;
\item \label{Ai} 
for $i=1,2$ set $A_i:=\psi(B_i)\subset Y$;
$A_1$ and $A_2$ are disjoint smooth surfaces, and
$A_1$ is a nef divisor;
\item \label{A2} if $a=0$, then $Y\cong\pr^1\times\pr^2$ and $A_i\cong\{pt\}\times\pr^2$
 for $i=1,2$; if $a\in\{1,2\}$, then $A_2$ is the negative section of  $\xi\colon Y\to\pr^2$ (namely $\ma{N}_{A_2/Y}\cong\ol_{\pr^2}(-a)$); in any case $A_2$ is a section of $\xi$ and $\xi_{|A_1}\colon A_1\to\pr^2$ is finite;
\item \label{Bi} for $i=1,2$ set   $T_i:=\psi^{-1}(A_i)\subset X_2$; 
 $B_i$ is a section of $\psi_{|T_i}\colon T_i\to A_i$, for $i=1,2$.
\end{enumerate}
\end{thm}
\begin{proof}
The existence of the diagram, together with properties ($\ref{g}$) and ($\ref{blowup}$), and the additional fact that $\psi\circ f\colon X\to Y$ is a conic bundle,
are shown in \cite[Th.~1.1 and its proof, in particular 3.3.15-3.3.17]{codim}.  Then applying \cite[Prop.~3.5(1)]{eleonora} to $\psi\circ f$ we get ($\ref{Bi}$).
Finally ($\ref{Y}$) is \cite[Prop.~1.2(a)]{M_R}, and ($\ref{Ai}$) and ($\ref{A2}$) are proved in \cite[proof of Prop.~1.2]{M_R}.
\end{proof}
We note that after ($\ref{Ai}$) and ($\ref{A2}$), the role of the two surfaces $A_1$ and $A_2$ is not symmetric if $a>0$.

\medskip

In the toric case, the classification is already known, and
 relies on Batyrev's classification of toric Fano $4$-folds \cite{bat2}, see \cite[Prop.~5.1]{M_R}. We follow the notation of \cite{bat2}.
\begin{proposition} \label{toric_case}
There are four toric Fano $4$-folds $X$ with $\rho_X=5$ and $\delta_X=3$; they are the $4$-folds $K_1$, $K_2$, $K_3$, and $K_4$.
\end{proposition}
\subsection{Proof of Theorem \ref{main}}
We keep the same notation as in Th.~\ref{factorization}.
\begin{step}\label{Z}
We can assume that there exists a commutative diagram:
$$\xymatrix{{X_2}\ar[r]^\psi\ar[d]_g&Y\ar[d]^{\xi}\\
Z\ar[r]^{\ph} &{\pr^2} 
} $$
where $\ph\colon Z\to\pr^2$ is a $\pr^2$-bundle, $g$ is the blow-up along a section $S_3\subset Z$ of $\ph$, $E:=\Exc(g)$ is a section of $\psi$, and $E\cap (B_1\cup B_2)=\emptyset$.
\end{step}
\begin{proof}
Consider the natural factorization of $f$ as a sequence of two blow-ups:
$$X\stackrel{f_1}{\la} X_{1}\stackrel{f_2}{\la} X_2$$
where $f_2$ is the blow-up of $B_2$ and $f_1$ is the blow-up of the transform  of $B_1$.

Let us consider the morphism $\zeta:=\xi\circ \psi\circ f_2\colon X_1\to\pr^2$. 
Since both $\psi$ and $\xi$ are smooth by Th.~\ref{factorization}($\ref{Y}$)-($\ref{g}$), the composition $\xi\circ \psi \colon X_2\to\pr^2$ is smooth. Moreover, since $A_2\subset Y$ is a section of $\xi$, and the center $B_2$ of the blow-up $f_2\colon X_1\to X_2$ is a section over $A_2$ (see Th.~\ref{factorization}($\ref{A2}$) and ($\ref{Bi}$)), we conclude that $B_2$ is a section of  $\xi\circ \psi \colon X_2\to\pr^2$. This implies that $\zeta\colon X_1\to\pr^2$
is smooth.

We show that $-K_{X_1}$ is 
$\zeta$-ample. Let $C\subset X_1$ be an irreducible curve such that $-K_{X_1}\cdot C\leq 0$. 
If $\w{C}\subset X$ is its transform, we have $-K_X\cdot\w{C}>0$ which implies that $\Exc(f_1)\cdot \w{C}<0$, hence
 $C\subset f_1(\Exc(f_1))$, $f_2(C)\subset B_1$, and $\psi(f_2(C))\subset A_1$. Since $\psi$ is finite on $B_1$ and $\xi$ is finite on $A_1$ by Th.~\ref{factorization}($\ref{Bi}$) and ($\ref{A2}$), we conclude that $\zeta(C)$ is a curve. This shows that $-K_{X_1}$ is positive on every curve contracted by $\zeta$. Being $X$ Fano, the cone $\NE(X)$ is closed and polyhedral, and this easily implies that $\NE(X_1)$ is closed. By the relative Kleiman's criterion, 
 $-K_{X_1}$ is 
$\zeta$-ample. 

Hence $\zeta\colon X_1\to\pr^2$
is a smooth contraction of relative Picard number $3$ with $-K_{X_1}$ relatively ample, and
every fiber of $\zeta$ is isomorphic to the del Pezzo surface $S$ with $\rho_S=3$, the blow-up of $\pr^2$ in two points. 

If $i\colon S\hookrightarrow X_1$ is the inclusion of a fiber, the pushforward of $1$-cycles $i_*\colon \N(S)\to\N(X_1)$ is injective, and yields and isomorphism $\N(S)\cong\ker\zeta_*$. It is clear that $i_*(\NE(S))\subseteq\NE(\zeta)$; conversely it follows from \cite[Prop.~1.3]{wisndef} that equality holds, so that every contraction of the fiber $S$ extends to a global contraction of $X_1$ over $\pr^2$.
Therefore the sequence of elementary contractions:
$$\xymatrix{S\ar[r]&{\mathbb{F}_1}\ar[d]\ar[r]&{\pr^1}\ar[d]\\
&{\pr^2}\ar[r]&
\{pt\}}$$
yields a corresponding factorization of $\zeta$:
$$\xymatrix{{X_1}\ar[r]^{f_2'}&{X_2'}\ar[d]_g\ar[r]^{\psi'}&{Y'}\ar[d]^{\xi'}\\
&{Z}\ar[r]^{\ph}&{\pr^2}}$$
We have:
\begin{enumerate}[$\bullet$]
\item
$\xi'\colon Y'\to\pr^2$ and $\psi'\colon X_2' \to Y'$ are  $\pr^1$-bundles, and $\ph\colon Z\to\pr^2$ is a  $\pr^2$-bundle;
\item $g$ is the blow-up of a smooth surface $S_3\subset Z$ which is a section of $\ph$;
\item
 $f_2'$ is the blow-up of a smooth surface $B_2'\subset X_2'$ which is a section of $\ph\circ g\colon X_2'\to \pr^2$, and is disjoint from $E:=\Exc(g)$;
\item
 $E$ is a section of $\psi'\colon X_2'\to Y'$.
\end{enumerate}
Notice also that the center of the blow-up $f_1\colon X\to X_1$ cannot meet any $(-1)$-curve in the fiber $S$, otherwise $X$ would not be Fano. Hence $B_1':=f_2'(f_1(\Exc(f_1)))$ is disjoint from $E$.

Finally, using \cite[Prop.~1.2(a)]{M_R}, we still get $Y'\cong\pr_{\pr^2}(\ol\oplus\ol(a))$ with $a\in\{0,1,2\}$, so that we can replace the original factorization of $\zeta$ with the new one keeping all the previous properties.
\end{proof}
Set $S_i:=g(B_i)\subset Z$ for $i=1,2$.
Then $S_1$, $S_2$, and $S_3$ are pairwise disjoint smooth surfaces, and 
 $X$ is the blow-up of $Z$ along $S_1\cup S_2\cup S_3$. We set $Z_p:=\ph^{-1}(p)$ for every $p\in\pr^2$. Moreover we denote by $d$ the degree of the finite morphism $\xi_{|A_1}\colon A_1\to\pr^2$
(see Th.~\ref{factorization}$(\ref{A2})$). 
\begin{step}\label{d}
 $S_2$ is a section of $\ph$, and $\ph_{|S_1}\colon S_1\to\pr^2$ is finite of degree $d$.
\end{step}
\begin{proof}
For $i=1,2$, since $B_i$ is a section over $A_i$ by Th.~\ref{factorization}$(\ref{Bi})$, the degree of $S_i$ over $\pr^2$ is equal to the degree of $A_i$ over $\pr^2$; in particular $S_2$ is a section of $\ph$ by Th.~\ref{factorization}$(\ref{A2})$.
\end{proof}
\begin{step}\label{points}
The points  $(S_1\cup S_2\cup S_3)\cap Z_p$ (with the reduced structure) are in general linear position in $Z_p\cong\pr^2$, for every $p\in\pr^2$.

{\em Indeed, if
 there were three of them on a line $\ell$, the transform of $\ell$ in $X$ would have non-positive anticanonical degree.}
\end{step}
\begin{step}\label{K}
If $d=1$, then 
 $X$ is toric, and it is one of the $4$-folds $K_1$, $K_2$, $K_3$, or $K_4$, in the notation of \cite{bat2}.
\end{step}
\begin{proof}
If $d=1$, then
$\ph\colon Z\to\pr^2$ has three pairwise disjoint sections $S_i$, which 
are fiberwise in general linear position, 
by Steps \ref{Z}, \ref{d}, and \ref{points}. 
This implies that $Z\cong \pr_{\pr^2}(\ma{L}_1\oplus\ma{L}_2\oplus\ma{L}_3)$ in such a way that the three sections $S_i$ correspond to the projections
$\ma{L}_1\oplus\ma{L}_2\oplus\ma{L}_3\twoheadrightarrow \ma{L}_i$.
This means that $Z$ is a toric $4$-fold, and that  $S_1$, $S_2$, and $S_3$ are invariant for the torus action, so that $X$ is toric. Then the statement follows from Prop.~\ref{toric_case}.
\end{proof}
From now on we assume  that  $d\geq 2$, in particular this implies that
$a\in\{1,2\}$
by Th.~\ref{factorization}($\ref{A2}$). 

For $q_1,q_2\in Z_p$ distinct points, we denote by $\overline{q_1q_2}\subset Z_p$ the line spanned by $q_1$ and $q_2$.
\begin{step}\label{H}
Let $H\subset Z$ be the relative secant variety of $S_1$ in $Z$, namely the
closure in $Z$ of the set:
$$\bigl\{\overline{q_1q_2}\subset Z_p\,|\,q_1,q_2\in S_1\cap Z_p,q_1\neq q_2,p\in\pr^2\bigr\}.$$
 For $p$ general, we have $|S_1\cap Z_p|=d\geq 2$, so that  $H$ is non-empty.
 It is not difficult to see that
$\dim H=3$, and Step \ref{points}  implies that $H\cap(S_2\cup S_3)=\emptyset$.
\end{step}

Recall that $T_2=\psi^{-1}(A_2)\subset X_2$ and that $E=\Exc(g)\subset X_2$ (see Th.~\ref{factorization}($\ref{Bi}$) and Step \ref{Z}).
\begin{step}\label{D}
Set $D:=g(T_2)\subset Z$. Then $T_2\cong D\cong\pr^1\times\pr^2$, and $D\cap Z_p$ is a line in $Z_p$ for every $p\in\pr^2$. Moreover $D$ contains $S_2$ and $S_3$,  while $D\cap S_1=\emptyset$.
\end{step}
\begin{proof}
We have a commutative diagram:
$$\xymatrix{{T_2}\ar[d]_{g_{|T_2}}\ar[r]^{\ \psi_{|T_2}}&
{A_2}\ar[d]^{\xi_{|A_2}}\\
D\ar[r]_{\ph_{|D}}&{\pr^2}
}$$
where the vertical  maps are isomorphisms by Step \ref{Z} and Th.~\ref{factorization}($\ref{A2}$), and the  horizontal  maps are $\pr^1$-bundles.

We also have $S_3=g(E\cap T_2)\subset D$; moreover $B_2\subset T_2$ by Th.~\ref{factorization}($\ref{Ai}$), hence $S_2=g(B_2)\subset D$.
By Steps \ref{Z} and \ref{d}, $S_2$ and $S_3$ are disjoint sections of the $\pr^1$-bundle $\ph_{|D}\colon D\to\pr^2$; this implies that  $D\cong\pr_{\pr^2}(\ol\oplus\ol(c))$ for some $c\in\Z$.

Now we have $H\cap D\neq\emptyset$, because both  contain a line in $Z_p$, so that $H\cap D$ yields a non-zero effective divisor in $D$. On the other hand, this divisor is disjoint from both sections $S_2$ and $S_3$ by Step \ref{H}. This easily implies that $c=0$ and $D\cong\pr^1\times\pr^2$.

Finally, we note that for every $p\in\pr^2$, $D\cap Z_p$ is the line spanned by the points $S_2\cap Z_p$ and $S_3\cap Z_p$, so that $D\cap S_1=\emptyset$ by Step \ref{points}.
\end{proof}
We denote by $L\in\Pic(Y)$ the tautological line bundle for $Y=\pr_{\pr^2}(\mathcal{O}\oplus \mathcal{O}(a))$; note that $L$ is nef and big, and $L_{|A_2}\cong\ol_{\pr^2}$ by Th.~\ref{factorization}($\ref{A2}$).
\begin{step}\label{normal}
There exists $b\in\Z$ such that ${\ma{N}}_{S_3/Z}^{\vee}\cong\mathcal{O}_{\pr^2}(b)\oplus \mathcal{O}_{\pr^2}(a+b)$ 
and $\ma{N}^{\vee}_{E/X_2}\cong L\otimes\xi^*(\ol_{\pr^2}(b))\in\Pic(Y)$. 
\end{step}
\begin{proof}
By Step \ref{Z} we have a commutative diagram:
$$\xymatrix{{E}\ar[d]\ar[r]^{\psi_{|E}}& Y\ar[d]^{\xi}\\
{S_3}\ar[r]^{\ph_{|S_3}}&{\pr^2}
}$$
where the horizontal maps are isomorphisms, and the vertical maps are $\pr^1$-bundles.
Since $g\colon X_2\to Z$ is the blow-up of $S_3$, using Th.~\ref{factorization}($\ref{Y}$) we get:
$$\mathbb{P}_{S_3}({\ma{N}}_{S_3/Z}^{\vee})\cong E\cong Y\cong \mathbb{P}_{\mathbb{P}^{2}}(\mathcal{O}\oplus \mathcal{O}(a)),$$ 
therefore ${\ma{N}}_{S_3/Z}^{\vee}\cong\mathcal{O}_{\pr^2}(b)\oplus \mathcal{O}_{\pr^2}(a+b)$, with $b\in \mathbb{Z}$. Moreover $\ma{N}^{\vee}_{E/X_2}$ is the tautological line bundle of $\mathbb{P}_{\mathbb{P}^{2}}(\mathcal{O}(b)\oplus \mathcal{O}(a+b))$, which gives the statement.
\end{proof}
\begin{step}\label{b=0}
We have $b=0$, $X_2\cong \pr_Y(\ol\oplus{L})$,
 and $E$ corresponds (as a section of $\psi$) to the projection $\ol\oplus{L}\twoheadrightarrow\ol$.
\end{step}
\begin{proof}
Let $\ma{E}$ be a rank $2$ vector bundle on $Y$ such that $X_2=\pr_Y(\ma{E})$.
We know by Step \ref{Z} that $E$ is a section of $\psi$; this section corresponds to a surjection $\sigma\colon \mathcal{E}\twoheadrightarrow \mathcal{F}$ with $\mathcal{F}\in\Pic(Y)$, and up to replacing $\mathcal{E}$ with $\mathcal{E}\otimes \ma{F}^{\vee}$ we may assume that $\ma{F}=\mathcal{O}_Y$, so that $\ker\sigma\cong \ma{N}_{E/X_2}^\vee\cong L\otimes\xi^*(\ol_{\pr^2}(b))$  by Step \ref{normal}.  We  obtain the following exact sequence over $Y$:
\begin{equation}\label{sequence}
0\longrightarrow \ker\sigma\longrightarrow \mathcal{E}\longrightarrow \mathcal{O}_Y\longrightarrow 0.\end{equation}

Now let us consider $A_2\subset Y$; 
we have $\ker\sigma_{|A_2}\cong 
L_{|A_2}\otimes\xi^*(\ol_{\pr^2}(b))_{|A_2}\cong\ol_{\pr^2}(b)$, so by restricting to $A_2$ the above exact sequence  we get:
$$0\longrightarrow \ol(b)\longrightarrow \mathcal{E}_{|A_2}\longrightarrow \mathcal{O}\longrightarrow 0.$$
On the other hand $\pr_{A_2}(\mathcal{E}_{|A_2})=T_2\cong \pr^2\times\pr^1$ by Step \ref{D}, and we deduce that $b=0$ and $\ker\sigma\cong L$.

We note that $Y$ is Fano, and $L$ is nef and big, so that $-K_Y+L$ is ample.
Therefore $\text{Ext}^1(\mathcal{O}_Y,\ker\sigma)\cong H^1(Y,L)=H^1(Y,K_Y-K_Y+L)=0$ by Kodaira vanishing, 
hence  the  sequence \eqref{sequence} splits, so that $\mathcal{E}\cong\mathcal{O}_Y\oplus L$.
\end{proof}
\begin{step}\label{section}
There exists a section $K$ of $\psi\colon X_2\to Y$ containing $B_1$ and disjoint from $E$.
\end{step}
\begin{proof}
By Th.~\ref{factorization}($\ref{Bi}$) and Step \ref{b=0}
we have   $B_1\subset T_1=\pr_{A_1}(\ol_{A_1}\oplus L_{|A_1})$ and 
$B_1$ is a section of  $\psi_{|T_1}\colon T_1\to A_1$. Moreover by Steps \ref{Z} and \ref{b=0} we deduce that
$T_1\cap E$ is another section of  $\psi_{|T_1}$, disjoint from $B_1$, and 
corresponding to the projection
$\ol_{A_1}\oplus L_{|A_1}\to\ol_{A_1}$.
Then it is not difficult to see that $B_1$ corresponds, as a section, to a surjection $\tau\colon\ol_{A_1}\oplus L_{|A_1}\twoheadrightarrow L_{|A_1}$.

Let us consider the restriction $r\colon \text{Hom}(\ol_Y\oplus L,L)\to \text{Hom}(\ol_{A_1}\oplus L_{|A_1},L_{|A_1})$.
We have
$$
\text{Hom}(\ol_{A_1}\oplus L_{|A_1},L_{|A_1})\cong \text{Hom}(L_{|A_1}^{\vee}\oplus \ol_{A_1},\ol_{A_1})\cong 
 H^0({A_1}, L_{|A_1})\oplus H^0({A_1},\ol_{A_1}),
$$
and similarly $\text{Hom}(\ol_{Y}\oplus L,L)\cong H^0(Y, L)\oplus H^0(Y,\ol_{Y})$. Since the restriction $H^0(Y,\ol_Y)\to H^0(A_1,\ol_{A_1})$ is an isomorphism, $r$ is surjective if the restriction $H^0(Y, L)\to H^0(A_1, L_{|A_1})$ is. 

We have an exact sequence of sheaves on $Y$:
$$0\la L\otimes\ol_Y(-A_1)\la L\la L_{|A_1}\la 0.$$
Since $A_1\cap A_2=\emptyset$ by Th.~\ref{factorization}($\ref{Ai}$),
and $A_1$ has degree $d$ over $\pr^2$, it is not difficult to see that $\ol_Y(A_1)= L^{\otimes d}$ in $\Pic(Y)$.
Then, using Serre duality and Kawamata-Viehweg vanishing:
$$H^1(Y, L\otimes\ol_Y(-A_1))=H^1(Y,L^{\otimes (1-d)})=H^2(Y,K_Y\otimes L^{\otimes (d-1)})=0$$
because $d\geq 2$ and $L$ is nef and big.

We conclude that $r$ is surjective, so that $\tau$ extends to a morphism $\overline{\tau}\colon \ol_Y\oplus L\to L$.

We show that $\overline{\tau}$ is surjective. By contradiction, suppose that $\im\overline{\tau}\subsetneq L$; then $\im\overline{\tau}\cong L\otimes\ol_Y(-D_0)$ with $D_0$ a non-zero effective divisor, and $\ker\overline{\tau}\cong\ol_Y(D_0)$. 

We have $\Hom(L,L\otimes\ol_Y(-D_0))=0$, hence $\ker\overline{\tau}\supseteq\{0\}\oplus L$. On the other hand $\Hom(\ol_Y(D_0),\ol_Y)=0$, hence $\ker\overline{\tau}\subseteq \{0\}\oplus L$. We conclude that $\ker\overline{\tau}= \{0\}\oplus L$ and $\overline{\tau}$ factors through the projection $\ol_Y\oplus L\twoheadrightarrow \ol_Y$. Then the same happens by restricting to $A_1$, which is impossible, because $\tau$ is surjective.

Thus we have a surjection $\overline{\tau}\colon\ol_Y\oplus L\twoheadrightarrow L$ which yields a section $K\subset X_2$ extending $B_1$.

We show that $K\cap E=\emptyset$. Let us consider the projection $\ol_Y\oplus L\twoheadrightarrow L$ and the corresponding section $K'\subset X_2$. Since $E$ corresponds to the projection onto the other summand, we have $K'\cap E=\emptyset$. On the other hand, it is easy to check that $K\sim K'$ in $X_2$, hence for every curve $C\subset E$ we have $K\cdot C=0$. Since $K\neq E$, this implies that $K\cap E=\emptyset$.
\end{proof}
\begin{step}\label{d=2}
We have $d=2$ and $H=g(K)$.
\end{step}
\begin{proof}
Consider $g(K)\subset Z$, so that $K\cong g(K)$ and 
$g(K)\supset S_1$ by Step \ref{section}. If $p\in\pr^2$ is general, 
then $g^{-1}(Z_p)\cong\mathbb{F}_1$, and $K\cap g^{-1}(Z_p)$ is a section of $\mathbb{F}_1\to\pr^1$, disjoint from the $(-1)$-curve $E\cap g^{-1}(Z_p)$.
Thus $g(K)\cap Z_p$ is a line in $Z_p\cong\pr^2$, and this line contains the $d$ points $S_1\cap Z_p$. Since these points  are in general linear position (see Step \ref{points}), we conclude that $d=2$.  We also deduce that
$g(K)=H$ (see Step \ref{H}). 
 \end{proof}
\begin{step} \label{S3}
We have $Z\cong\pr_{\pr^2}(\mathcal{O}^{\oplus 2}\oplus \mathcal{O}(a))$, 
and under the isomorphism $D\cong\pr^1\times \pr^2$ one has $\ol_Z(D)_{|\{\text{pt}\}\times\pr^2}\cong\ol_{\pr^2}(-a)$.
\end{step}
\begin{proof}
 We know by Steps \ref{Z}, \ref{normal} and \ref{b=0} that $S_3$ is a section of $\ph\colon Z\to \mathbb{P}^{2}$ with conormal bundle $\ol\oplus\ol(a)$. 
As in the proof of Step \ref{b=0}, using that $H^{1}(\mathbb{P}^{2}, \mathcal{O}_{\mathbb{P}^{2}}(c))=0$ for every $c\in \mathbb{Z}$, one shows that $Z\cong\pr_{\pr^2}(\mathcal{O}\oplus \mathcal{O}\oplus \mathcal{O}(a))$.

Recall from Step \ref{D} that $D\cong\pr^1\times\pr^2$ and $S_3\cong\{\text{pt}\}\times\pr^2$. Thus
 $\ma{N}_{S_3/Z}\cong\ol_{\pr^2}\oplus\ol_{\pr^2}(-a)$ and $\ma{N}_{S_3/D}\cong\ol_{\pr^2}$; using the exact sequence on $S_3$:
$$0\la \ma{N}_{S_3/D}\la \ma{N}_{S_3/Z}\la \ma{N}_{D/Z|S_3}\la 0$$
we conclude that $\ma{N}_{D/Z|S_3}\cong\ol_{\pr^2}(-a)$.
\end{proof}
\begin{step}\label{Ex1}
If $a=1$, then $X$ is the $4$-fold of Example \ref{e1}.
\end{step}
\begin{proof}
Set $a=1$. Applying Step \ref{S3} we have
 $Z\cong\pr_{\pr^2}(\mathcal{O}^{\oplus 2}\oplus  \mathcal{O}(1))$, that is the blow-up of $\pr^4$ along a line, with exceptional divisor $D$; moreover $S_2$ and $S_3$ are non-trivial fibers of the blow-up $Z\to\pr^4$. The image of $H$ (see Steps \ref{H} and \ref{d=2}) in $\pr^4$ is a hyperplane, and
the image of $S_1$ is a smooth quadric surface. This is Example \ref{e1}.
\end{proof}
\begin{step}\label{Ex2}
If $a=2$, then $X$ is the $4$-fold of Example \ref{e2}.
\end{step}
\begin{proof}
Set $a=2$. Using Step \ref{S3} one has $Z\cong\pr_{\pr^2}(\mathcal{O}^{\oplus 2} \oplus \mathcal{O}(2))$, then 
there is a divisorial contraction $Z\to W$ sending $D$ to a curve; moreover $S_2,S_3\subset D$ are fibers of this contraction. 

Consider the divisor $H\subset Z$ (see Step \ref{H}), and let $\ol_W(1)$ be the ample generator of $\Pic(W)$. Using Steps \ref{section} and \ref{d=2}, and Th.~\ref{factorization}($\ref{Y}$) we have $H\cong K\cong Y=\pr_{\pr^2}(\mathcal{O}\oplus \mathcal{O}(2))$; and it is not difficult to see that $\ol_Z(H)$
  is the pullback of $\ol_W(1)$. Moreover, $S_1\subset H$ is disjoint from the negative section $D\cap H$ (see Step \ref{D}) and has degree $2$ over $\pr^2$; this implies that $S_1\in|\ol_Z(2H)_{|H}|$ in $H$. Since $Z$ is Fano and $\ol_Z(H)$ is nef, we have $h^1(Z,H)=h^1(Z,K_Z-K_Z+H)=0$ by Kodaira vanishing, therefore the restriction $H^0(Z,\ol_Z(2H))\to H^0(H,\ol_Z(2H)_{|H})$ is surjective. We conclude that $S_1$
is a complete intersection of elements in the linear systems $|H|$ and $|2H|$, and  this is Example \ref{e2}.
\end{proof}
This concludes the proof of Th.~\ref{main}.
\begin{remark}
A posteriori, we see that the varieties $X_1$, $X_2$, $Z$ and $Y$ are always toric; moreover it is not difficult to see that there is always a choice of 
the ordering of the three blow-ups $X\to Z$ such that $X_1$ is toric and Fano. 

More precisely, in the two non-toric cases one can choose an ordering of the blow-ups in such a way that,
 following the notation of \cite{bat2} for toric Fano $4$-folds:
\begin{enumerate}[--]
\item
in Example \ref{e1},  $X_1$ is $H_4$ and $X_2$ is $D_8$;
\item in Example \ref{e2}, $X_1$ is $H_1$ and $X_2$ is $D_2$.
\end{enumerate}
\end{remark}
\begin{remark}[conic bundles of Fano manifolds]
The varieties $X$ of Examples \ref{e1} and \ref{e2} give new examples of Fano varieties with a conic bundle $X\to Y$ such that $\rho_X-\rho_Y=3$. It is shown in \cite[Th.~1.1]{eleonora} that if $X$ is a Fano manifold (of arbitrary dimension) which is not a product of varieties of smaller dimension, and has a conic bundle $X\to Y$, then $\rho_X-\rho_Y\leq 3$.
\end{remark} 
Given a conic bundle $h\colon X\to Y$, let $\bigtriangleup:=\{y\in Y| \ h^{-1}(y) \ \text{is singular}\}$ be its \textit{discriminant divisor}. As a consequence of our main result, we find an explicit description of the discriminant divisors of conic bundles encoded by the varieties of Th.~\ref{main}. See also \cite[Cor.~3.4]{M_R} for some partial results in this direction.
\begin{corollary} Let $X$ be a Fano 4-fold with $\rho_X=5$, admitting a conic bundle $h\colon X\to Y$ such that $\rho_X-\rho_Y=3$. Denote by $\bigtriangleup$ the discriminant divisor of $h$. Then one of the following holds:
\begin{enumerate}[$(i)$]
\item $\bigtriangleup\cong \mathbb{P}^2 \sqcup \mathbb{P}^2$, and $X$ is toric of combinatorial type $K$;
\item $\bigtriangleup\cong (\mathbb{P}^1\times \mathbb{P}^1) \sqcup \mathbb{P}^2$, and $X$ is the variety of Example \ref{e1};
\item  $\bigtriangleup\cong S \sqcup \mathbb{P}^2$ where $S$ is a del Pezzo surface of degree 2, and $X$ is the variety of Example \ref{e2}.
\end{enumerate}
\begin{proof} We first show that $\delta_X=3$. In view of \cite[Th.~1.1]{M_R} we are left to analyze the case in which $X\cong S_1\times S_2$ with $S_i$ del Pezzo surfaces. In this situation, $h$ is induced by a conic bundle on one of the two del Pezzo surfaces $S_i$, say $S_1$, and $Y\cong\pr^1\times S_2$
(see e.g.\ 
the proof of \cite[~Th. 4.2(1)]{eleonora}). Being $\rho_Y=\rho_X-3=2$, we conclude that $S_2\cong\pr^2$, $\rho_{S_1}=4$, and
 $X\cong S_1\times\pr^2=K_4$, thus $\delta_X=3$. 

We observe that $h\colon X\to Y$ satisfies all the properties $(\ref{Y})$-$(\ref{Bi})$ listed in Th.~\ref{factorization}. Indeed, by \cite[Prop.~3.5(1) and Prop.~4.2(2)]{eleonora} we may take a factorization for $h$ such that $(\ref{g})$, $(\ref{blowup})$, and $(\ref{Bi})$ hold. All the remaining properties follow by arguing as in the proof of Th.~\ref{factorization}. This implies that we can run the arguments of the proof of Th.~\ref{main} replacing $\psi \circ f$ by $h$. Let us keep the notation as in that theorem.

Then $\bigtriangleup=A_1\sqcup A_2$, and using Step \ref{Z} we have $A_i\cong B_i\cong S_i$ for $i=1,2$. At this point, Steps \ref{K}, \ref{Ex1}, \ref{Ex2} and their proofs give respectively $(i)$, $(ii)$, and $(iii)$, hence the statement. 
\end{proof}
\end{corollary}
\begin{proof}[Proof of Cor.~\ref{rho2}]
If $\delta_X\leq 3$, then $\rho_X\leq \delta_X+\dim\N(D,X)\leq 5$, and the statement follows from Th.~\ref{main}.

If instead $\delta_X\geq 4$, by \cite[Th.~1.1]{codim} we have $X\cong S\times T$ where $S$ and $T$ are del Pezzo surfaces. We can assume that $\pi(D)=T$, where $\pi\colon X\to T$ is the projection. Let us consider the pushforward $\pi_{*}\colon \N(X)\to\N(T)$. Then $\pi_*(\N(D,X))=\N(T)$; on the other hand $\pi$ cannot be finite on $D$, so that $\N(D,X)\cap\ker\pi_*\neq 0$, and $\dim\N(T)<\dim\N(D,X)=2$. We conclude that $\rho_T=1$, hence $T\cong\pr^2$, and being $\delta_{X}=\rho_{S}-1$ (cf.\ \cite[Ex.~3.1]{codim}) we get $\rho_S\geq 5$ and $\rho_X\geq 6$.
\end{proof}
\begin{proof}[Proof of Cor.~\ref{31}]
Let us take the push-forward $\sigma_*\colon\N(X)\to\N(X')$. We have $\sigma_*(\N(\Exc(\sigma),X))=\R[\sigma(\Exc(\sigma))]$, and $\dim\ker\sigma_*=1$, thus $\dim\N(\Exc(\sigma),X)=2$. By Cor.~\ref{rho2}, $X$ is one of the $4$-folds appearing in Th.~\ref{main}, so we only have to check which of these varieties admit a contraction as in the statement.
We already know by Th.~\ref{toru} that the toric $4$-fold $K_3$ and Example \ref{e1} do. 

It is easy to check from \cite{bat2}, using the primitive relations of the varieties $K_i$, that $K_2$ and $K_4\cong\pr^2\times S_4$ do not have any elementary divisorial contraction such that the image of the exceptional divisor is a curve, while $K_1$ does, and satisfies the statement.

Concerning Example \ref{e2}, keeping the same notation as in the example, we have $D\cong \pr^1\times\pr^2$ and $\ma{N}_{D/Z}\cong\ol_{\pr^1\times\pr^2}(1,-2)$.
If $\w{D}\subset X$ is the transform of $D$, it is not difficult to check that 
$\w{D}\cong \pr^1\times\pr^2$ and $\ma{N}_{\w{D}/X}\cong\ol_{\pr^1\times\pr^2}(-1,-2)$.
Then $-2K_X+\w{D}$ is nef, and has intersection zero only with the curves in $\{pt\}\times\pr^2\subset\w{D}$, so the classes of these curves belong to an extremal ray of $\NE(X)$ which gives the desired contraction.
\end{proof}

\bigskip

\section{The examples}\label{explicit}
\subsection{Example \ref{e2}}\label{details}
\noindent The $4$-fold $Z:=\pr_{\pr^2}(\ol^{\oplus 2}\oplus \ol(2))$ has two contractions, the $\pr^2$-bundle $\ph\colon Z\to\pr^2$, and a divisorial contraction $Z\to W:=\pr(1,1,1,2,2)$. We note that $\Pic(W)$ is generated by a very ample line bundle $\ol_W(1)$ which embeds $W$ in $\pr^7$ as  the cone over the Veronese surface, with vertex a line; the birational morphism $Z\to W$ sends the exceptional divisor $D$ to this line. 

Let  $H\subset Z$ be the pullback of a general element in $|\ol_W(1)|$; then  $H$ is a resolution of a cone over the Veronese surface (with vertex a point), and $H\cong\pr_{\pr^2}(\ol\oplus\ol(2))$. We have:
$$H\sim D+\ph^*\ol_{\pr^2}(2)\quad\text{and}\quad -K_Z= 3H+\ph^*\ol_{\pr^2}(1).$$

Let $F_1,F_2\subset D$ be two distinct fibers of the contraction $Z\to W$, so that $F_i\cong\pr^2$ are sections of $\ph\colon Z\to\pr^2$. In $D$ we have $F_i\in|{H}_{|D}|$.

Let $S\subset Z$ be a general complete intersection of elements in the linear systems $|{H}|$ and $|2{H}|$; notice that $\ph_{|S}\colon S\to\pr^2$ is finite of degree $2$.
By adjunction $-K_S=(-K_Z-3{H})_{|S}=\ph^*\ol_{\pr^2}(1)_{|S}$, and we deduce that $S$ is a del Pezzo surface of degree $2$.

The surfaces $S,F_1,F_2$ are pairwise disjoint in $Z$. Let
 $\sigma\colon X\to Z$ be the blow-up of $Z$ along $S$, $F_1$, and $F_2$.
\begin{lemma}\label{nef}
$-K_X$ has positive intersection with every curve in $X$.
\end{lemma}
\begin{proof}
Let $\w{D}\subset X$ be the transform of $D$, $E_i$ 
the exceptional divisor over $F_i$, 
and $E_0$ the exceptional divisor over $S$. 
Using that the surface $S$ has degree 2 over $\pr^2$, it is not difficult to see that there exists a unique $H_0\in|{H}|$ containing it; let $\w{H}_0\subset X$ be its transform. By the generality of $S$, $H_0$ is smooth and disjoint from $F_1$ and $F_2$. 

Let $\Gamma\subset X$ be an irreducible curve not contained in any of the divisors $\w{D}$, $\w{H}_0$,  $E_1$, or $E_2$. If $\sigma(\Gamma)$ is a point, then $-K_X\cdot\Gamma=1$. Otherwise, 
we set $\Gamma':=\sigma(\Gamma)\subset Z$, so that $\Gamma'$ is not contained in $D$, nor in $H_0$, nor in $F_i$.

For $i=1,2$ let $H_i\in|{H}|$ be a general element containing $F_i$,  so that these divisors do not contain $\Gamma'$. 
 Then $(\sigma^*H_i-E_i)\cdot\Gamma\geq 0$ for $i=0,1,2$,
 and we get
\begin{gather*}
-K_X\cdot \Gamma=\Bigl(\sigma^*(-K_Z)-\sum_{i=0}^2E_i\Bigr)\cdot\Gamma=
\Bigl(\sigma^*\bigl(3{H}+\ph^*\ol_{\pr^2}(1)\bigr)-\sum_{i=0}^2E_i\Bigr)\cdot\Gamma\\=
\sum_{i=0}^2 (\sigma^*{H}-E_i)\cdot\Gamma+\sigma^*\ph^*\ol_{\pr^2}(1)\cdot\Gamma\geq 0.
\end{gather*}

We show that the intersection is in fact positive. If $-K_X\cdot \Gamma=0$, then $\sigma^*\ph^*\ol_{\pr^2}(1)\cdot\Gamma=0$, so that $\Gamma'$ is contained in a fiber $Z_p$ of the $\pr^2$-bundle $\ph\colon Z\to\pr^2$. 
Let $i\in\{1,2\}$. We also have  
$(\sigma^*H_i-E_i)\cdot\Gamma=0$, so that the transform of $H_i$ in $X$ is disjoint from $\Gamma$. Note that $(H_i)_{|Z_p}$ is a line in $Z_p\cong\pr^2$; this means that $H_i$ meets $\Gamma'$ only at the point $z_i:=F_i\cap Z_p$, transversally. Thus $\Gamma'$ must be the line $\overline{z_1z_2}$ in $Z_p$. However this line is contained in the exceptional divisor $D$, indeed $D\cap Z_p$ is a line and contains both points $z_i$, against our assumptions on $\Gamma$.

\bigskip

Now we show that the restriction of $-K_X$ to the divisors $\w{D}$, $\w{H}_0$, $E_1$, $E_2$  is ample.

We have $\w{D}\cong\pr^1\times\pr^2$ with $\ma{N}_{\w{D}/X}\cong\ol_{\pr^1\times\pr^2}(-1,-2)$, so that $(-K_X)_{|\w{D}}=-K_{\w{D}}+\ma{N}_{\w{D}/X}\cong\ol_{\pr^1\times\pr^2}(1,1)$ is ample.

\medskip

Similarly, for $i=1,2$,  $E_i\cong\pr_{\pr^2}(\ol\oplus\ol(2))$ and $\ma{N}_{E_i/X}\cong -A-2B$, where $A$ is the negative section of $E_i\to\pr^2$ and $B$ is the pullback of $\ol_{\pr^2}(1)$. Moreover $-K_{E_i}=2A+5B$, so that $(-K_X)_{|E_i}=-K_{E_i}+\ma{N}_{E_i/X}\cong A+3B$ is ample.

\medskip

Finally we have $\w{H}_0\cong H_0\cong\pr_{\pr^2}(\ol\oplus\ol(2))$, 
and the negative section is $\w{H}_0\cap \w{D}$; we know that $(-K_X)_{|\w{H}_0\cap\w{D}}$ is ample.

If $\ell\subset \w{H}_0$ is a fiber of the $\pr^1$-bundle, then $\ell\cdot E_i=0$ for $i=1,2$, $\ell\cdot E_0=2$, and $-K_Z\cdot\sigma(\ell)=3$, thus $-K_X\cdot \ell=1$.

Note that $\NE(\w{H}_0)$ is generated by $[\ell]$ and by the class of a curve in 
$\w{H}_0\cap \w{D}$; we conclude that $(-K_X)_{|\w{H}_0}$ is ample.
\end{proof}
The following is a standard computation.
\begin{lemma}\label{formula}
Let $f\colon X\to Y$ be the blow-up of a smooth projective $4$-fold along a smooth irreducible surface $S$. Then
we have the following:
\begin{align*}
K_X^4&=K_Y^4-3(K_{Y|S})^2-2 K_S\cdot K_{Y|S}+c_2(\ma{N}_{S/Y})-K_S^2,\\
K_X^2\cdot c_2(X)&=K_Y^2\cdot c_2(Y)-12\chi(\ol_S)+2K_S^2-2K_S\cdot K_{Y|S}-2c_2(\ma{N}_{S/Y}),\\
\chi(X,-K_X)&=\chi(Y,-K_Y)-\chi(\ol_S)-\frac{1}{2}\bigl((K_{Y|S})^2+K_S\cdot K_{Y|S}
\bigr).
\end{align*}
\end{lemma}
\begin{proof}
Let $E\subset X$ be the exceptional divisor, set $\pi:=f_{|E}
\colon E\to S$, and let $F\subset E$ be a fiber of $\pi$.

We have $(f^*K_Y)^4=K_Y^4$, and for $i\in\{0,1,2,3\}$ one has $(f^*K_Y)^i\cdot E^{4-i}=
(\pi^*K_{Y|S})^i\cdot (E_{|E})^{3-i}$. 
This gives
$$(f^*K_Y)^3\cdot E=0\quad\text{and}\quad (f^*K_Y)^2\cdot E^2=-(K_{Y|S})^2.$$

In $H^4(E,\Z)$ we have $\sum_{i=0}^2(-1)^i\pi^*c_i(\ma{N}_{S/Y}^{\vee})(-E_{|E})^{2-i}=0$, which yields
$$E_{|E}^2=-(\pi^*c_1(\ma{N}_{S/Y}^{\vee}))\cdot E_{|E} -c_2(\ma{N}_{S/Y}^{\vee})F.$$
Recall also that $c_1(\ma{N}_{S/Y}^{\vee})=K_{Y|S}-K_S$. 
Using these formulas, we get
\begin{gather*}
f^*K_Y\cdot E^3=(K_{Y|S})^2-K_S\cdot K_{Y|S},\\
E^4=c_2(\ma{N}_{S/Y})+2K_S\cdot K_{Y|S}-(K_{Y|S})^2-K_S^2.
\end{gather*}
Finally we have $K_X^4=(f^*K_Y+E)^4=\sum_{i=0}^4\binom{4}{i}(f^*K_Y)^i\cdot E^{4-i}$, which yields the formula for $K_X^4$.

By \cite[Ex.~15.4.3]{fultonint} we have $c_2(X)=f^*c_2(Y)+j_*(\pi^*(K_S)-E_{|E})$ in $H^4(X,\Z)$, where $j\colon E\hookrightarrow X$ is the inclusion and $j_*\colon H^2(E,\Z)\to H^4(X,\Z)$ is the Gysin homomorphism. Using that $j_*\alpha\cdot \beta=\alpha\cdot\beta_{|E}$ for every $\alpha\in H^2(E,\Z)$ and $\beta\in H^4(X,\Z)$, a computation similar to the previous one  gives the formula for $K_X^2\cdot c_2(X)$. Finally, $\chi(X,-K_X)$ is given by the Riemann-Roch formula, which in this setting is as follows:
$$\chi(X,-K_X)=\frac{1}{12}\bigl(2K_X^2+K_X^2\cdot c_2(X)\bigr)+\chi(X,\ol_X).$$
\end{proof}
\begin{lemma}\label{e2fano}
The $4$-fold $X$ is Fano with $\rho_X=5$ and $\delta_X=3$. We also have:
\begin{gather*}
K_X^4=250,\quad K_X^2\cdot c_2(X)=172,\quad h^0(X,-K_X)=57,\\ b_3(X)=0,\quad b_4(X)=h^{2,2}(X)=13.\end{gather*}
 Moreover $h^1(X,{T}_X)=h^0(X,{T}_X)+6\geq 6$, where ${T}_X$ is the tangent bundle.
\end{lemma}
\begin{proof}
We have $K_Z^4=594$, $K_Z^2\cdot c_2(Z)=240$, and $\chi(Z,-K_Z)=120$; this can be computed using toric geometry, or see \cite[Table 4, n.~7]{bat2}. 

We compute $K_X^4$ using Lemma \ref{formula}.

We calculate the contribution of $F_i\cong\pr^2$. We have $K_{Z|F_i}\cong\ol_{\pr^2}(-1)$ and $\ma{N}_{F_i/Z}\cong\ol_{\pr^2}(-2)\oplus\ol_{\pr^2}$;  this gives that the blow-up along each surface $F_i$ makes $K^4$ decrease by $18$.

Now we compute the contribution of $S$. In $Z$ we have
$H^4=4$, as the degree of $W\subset\pr^7$ is equal to the degree of the Veronese surface. Moreover $H^3\cdot\ph^*\ol_{\pr^2}(1)=2$, because if $\ell\subset\pr^2$ is a line, the image of $\ph^{-1}(\ell)$ under the birational map $Z\to W$ is a quadric cone. Finally $H^2\cdot\ph^*\ol_{\pr^2}(1)^2=1$, because $H_{|Z_p}\cong\ol_{\pr^2}(1)$ for every fiber $Z_p$ of $\ph$. We get:
\begin{gather*}
(K_{Z|S})^2=K_Z^2\cdot 2H^2=\bigl(3H+\ph^*\ol_{\pr^2}(1)\bigr)^2\cdot 2 H^2=98,\\
K_{Z|S}\cdot K_S=\bigl(3H+\ph^*\ol_{\pr^2}(1)\bigr)\cdot \ph^*\ol_{\pr^2}(1)\cdot 2H^2=14.
\end{gather*}
Moreover $\ma{N}_{S/Z}=\ol_S(H)\oplus\ol_S(2H)$, so that $c_2(\ma{N}_{S/Z})=H_{|S}\cdot 2H_{|S}=2H^2\cdot 2H^2=16$. In the end the blow-up along $S$ makes $K^4$ decrease by $308$, and we get:
$$K_X^4=K_Z^4-36-308=250.$$
Together with Lemma \ref{nef}, this shows that $-K_X$ is nef and big, hence semiample by the base-point-free theorem, and hence ample. Therefore $X$ is a Fano $4$-fold, and $\rho_X=5$.

For $i=1,2$ the divisor $E_i$ has $\rho_{E_i}=2$, so that $\delta_X\geq 5-2=3$. On the other hand $X$ is not a product, so \cite[Th.~1.1]{codim} implies that $\delta_X\leq 3$, and we conclude that $\delta_X=3$.

The values of  $K_X^2\cdot c_2(X)$ and $h^0(X,-K_X)=\chi(X,-K_X)$ can be computed from the ones of $Z$ using Lemma \ref{formula}, as done for $K_X^4$.
Also the Hodge numbers of $X$ can be easily computed by the explicit description of $X$ as a blow-up.

Finally, since $X$ is Fano, by Nakano vanishing we have $h^i(X,T_X)=0$ for every $i\geq 2$, so that $\chi(X,T_X)=h^0(X,T_X)-h^1(X,T_X)$. We can compute  $\chi(X,T_X)$
from the other invariants 
using Riemann Roch, see for instance \cite[Lemma 6.25]{vb} for an explicit formula.
\end{proof}
\subsection{Numerical invariants}
Table \ref{table} gives some relevant invariants of the varieties listed in Th.~\ref{main}; $T$ is the tangent bundle, and $S_4$ is the del Pezzo surface with $\rho_{S_4}=4$, namely
the blow-up of $\pr^2$ in three non collinear points. The invariants for the toric cases are given in \cite{bat2}, and the ones for Example \ref{e2} are given in Lemma \ref{e2fano}. For 
 Example \ref{e1}, the invariants are computed as for Example \ref{e2}, using the same technique as in the proof of Lemma \ref{e2fano}.

\medskip

\stepcounter{thm}
\begin{table}[!htbp]\caption{Fano $4$-folds with $\rho=5$ and $\delta=3$}\label{table}
$$\begin{array}{|c|c|c|c|c|c|c|c|c|c|}
\hline\hline
\rule{0pt}{2.5ex}  \text{$4$-fold}   & b_3 & h^{2,2} & h^{1,3} & K^4 & K^2\cdot c_2 & h^0(-K) & \chi({T})& \\
\hline\hline

K_1 &  0 & 6 & 0 & 364 &  196   & 78 & 10 &\text{toric} \\

\hline

K_2 &  0 & 6 & 0 & 354 & 192 &   76 & 10& \text{toric} \\

\hline

K_3 & 0 & 6 & 0 & 334 & 184  &  72 & 10 &\text{toric} \\

\hline

K_4\cong\pr^2\times S_4 &  0 & 6 & 0 & 324 & 180 & 70 & 10 &\text{toric} \\

\hline

\text{Ex.\ \ref{e1}} &  0 & 7 & 0 & 253 & 166  & 57 & 3 & \text{non toric} \\

\hline

\text{Ex.\ \ref{e2}} & 0 & 13 & 0 & 250 &  172 & 57 & -6\ \ &\text{non toric} \\

\hline\hline
\end{array}$$
\end{table}
%\bibliographystyle{amsalpha}
%\bibliography{BiblioBreve}

\begin{thebibliography}{CCF19}

\bibitem[Bat99]{bat2}
V.~V. Batyrev, \emph{On the classification of toric {F}ano 4-folds}, J.\ Math.\
  Sci.\ (New York) \textbf{94} (1999), 1021--1050.

\bibitem[Cas12]{codim}
C.~Casagrande, \emph{On the {P}icard number of divisors in {F}ano manifolds},
  Ann.~Sci.~{\'E}c.~Norm.~Sup{\'e}r.\ \textbf{45} (2012), 363--403.

\bibitem[Cas13]{cdue}
\bysame, \emph{\noop{zzz}{N}umerical invariants of {F}ano 4-folds},
  Math.~Nachr.\ \textbf{286} (2013), 1107--1113.

\bibitem[CCF19]{vb}
C.~Casagrande, G.~Codogni, and A.~Fanelli, \emph{The blow-up of {$\pr^4$} at
  {$8$} points and its {F}ano model, via vector bundles on a del {P}ezzo
  surface}, Rev.\ M{\'a}t.\ Complut.\ \textbf{32} (2019), 475--529.

\bibitem[CD15]{minimal}
C.~Casagrande and S.~Druel, \emph{Locally unsplit families of large
  anticanonical degree on {F}ano manifolds}, Int.\ Math.\ Res.\ Not.\
  \textbf{2015} (2015), 10756--10800.

\bibitem[Ful98]{fultonint}
W.~Fulton, \emph{Intersection theory}, second ed., Springer, 1998.

\bibitem[MM81]{morimukai}
S.~Mori and S.~Mukai, \emph{Classification of {F}ano {$3$}-folds with $b_2\geq
  2$}, Manuscr.~Math.\ \textbf{36} (1981), 147--162, Erratum: {\bf 110} (2003),
  407.

\bibitem[MR19]{M_R}
P.~Montero and E.~A. Romano, \emph{A characterization of some {F}ano 4-folds
  through conic fibrations}, Int.\ Math.\ Res.\ Not., published online 11
  November, 2019.

\bibitem[Rom19]{eleonora}
E.~A. Romano, \emph{Non-elementary {F}ano conic bundles}, Collectanea
  Mathematica \textbf{70} (2019), 33--50.

\bibitem[Sat00]{sato}
H.~Sato, \emph{Toward the classification of higher-dimensional toric {F}ano
  varieties}, T{\^o}hoku Math.\ J.\ \textbf{52} (2000), 383--413.

\bibitem[Tsu10]{toru5}
T.~Tsukioka, \emph{{F}ano manifolds obtained by blowing up along curves with
  maximal {P}icard number}, Manuscripta Math.\ \textbf{132} (2010), 247--255.

\bibitem[Wi{\'s}91]{wisndef}
J.~A. Wi{\'s}niewski, \emph{On deformation of nef values}, Duke Math.\ J.\
  \textbf{64} (1991), 325--332.

\end{thebibliography}
\providecommand{\noop}[1]{}
\providecommand{\bysame}{\leavevmode\hbox to3em{\hrulefill}\thinspace}
\providecommand{\MR}{\relax\ifhmode\unskip\space\fi MR }
\providecommand{\MRhref}[2]{%
  \href{http://www.ams.org/mathscinet-getitem?mr=#1}{#2}
}
\providecommand{\href}[2]{#2}

\end{document}